\documentclass[12pt]{amsart}

\pdfoutput=1

\usepackage[utf8]{inputenc}
\usepackage{graphicx}
\usepackage{amsfonts}
\usepackage{amsmath}
\usepackage{amsthm}
\usepackage{amssymb}
\usepackage{pgfplots}
\pgfplotsset{compat=1.17}
\usepackage{graphicx}
\usepackage{float}
\usepackage{upgreek}
\usepackage{url}
\usepackage{siunitx}
\usepackage{booktabs}
\usepackage{enumitem}
\usepackage{multicol}
\usepackage{mathrsfs}
\usepackage{algpseudocode}
\usepackage{adjustbox}
\usepackage{algorithm}
\usepackage{mathtools}
\usepackage{tikz-cd}
\usepackage[font=small,labelfont=bf]{caption}
\usepackage{wrapfig}
\usepackage[margin=1.21in]{geometry}
\usepackage{datetime}
\usepackage[
    backend=biber,
    style=ieee,
    citestyle=numeric-comp,
    sorting=nyt,
    dashed=false
]{biblatex}
\usepackage{hyperref}
\usepackage{xurl}
\hypersetup{
    colorlinks=true,
    linkcolor=blue,
    citecolor=red,
    breaklinks=true
}

\newdateformat{monthyeardate}{\monthname[\THEMONTH] \THEYEAR}

\title[On a time-frequency blurring operator]{On a time-frequency blurring operator with applications in data augmentation}

\author{Simon Halvdansson}
\address{Department of Mathematical Sciences, Norwegian University of Science and Technology, 7491 Trondheim, Norway.}
\email{simon.halvdansson@ntnu.no}

\date{\monthyeardate\today}

\theoremstyle{plain}
\newtheorem{theorem}{Theorem}[section]
\newtheorem*{theorem*}{Theorem}
\newtheorem{lemma}[theorem]{Lemma}
\newtheorem{proposition}[theorem]{Proposition}

\theoremstyle{definition}

\theoremstyle{remark}

\newcommand{\C}{\mathbb{C}}
\newcommand{\R}{\mathbb{R}}

\newcommand{\supp}{\operatorname{supp}}

\makeatletter
\newcommand{\vast}{\bBigg@{4}}
\newcommand{\Vast}{\bBigg@{5}}
\makeatother

\DeclareFontFamily{U}{mathx}{\hyphenchar\font45}
\DeclareFontShape{U}{mathx}{m}{n}{
      <5> <6> <7> <8> <9> <10>
      <10.95> <12> <14.4> <17.28> <20.74> <24.88>
      mathx10
      }{}
\DeclareSymbolFont{mathx}{U}{mathx}{m}{n}
\DeclareFontSubstitution{U}{mathx}{m}{n}
\DeclareMathAccent{\widecheck}{0}{mathx}{"71}
\DeclareMathAccent{\wideparen}{0}{mathx}{"75}

\def\XXint#1#2#3{{\setbox0=\hbox{$#1{#2#3}{\int}$ }
		\vcenter{\hbox{$#2#3$ }}\kern-.6\wd0}}

\begin{document}
    \maketitle
    \begin{abstract}\vspace{-6mm}
    Inspired by the success of recent data augmentation methods for signals which act on time-frequency representations, we introduce an operator which convolves the short-time Fourier transform of a signal with a specified kernel. Analytical properties including boundedness, compactness and positivity are investigated from the perspective of time-frequency analysis. A convolutional neural network and a vision transformer are trained to classify audio signals using spectrograms with different augmentation setups, including the above mentioned time-frequency blurring operator, with results indicating that the operator can significantly improve test performance, especially in the data-starved regime.
    \vspace{0mm}
    \end{abstract}

    \renewcommand{\thefootnote}{\fnsymbol{footnote}}
    \renewcommand{\thefootnote}{\arabic{footnote}}
    
    \section{Introduction and motivation}
    In time-frequency analysis, functions are analyzed in the \emph{phase space} of time and frequency. The phase space representation of a function can be modified to e.g. change the pitch or mask out an unwanted hiss and a signal can then be synthesized back. In this paper, we investigate the action of \emph{blurring} or \emph{spreading} a function in phase space. More precisely, we convolve the short-time Fourier transform (STFT)
    \begin{align*}
        V_\varphi \psi(x, \omega) = \langle \psi, \pi(x, \omega) \varphi \rangle = \int_{\R^d} \psi(t) \overline{\varphi(t-x)}e^{-2\pi i \omega \cdot t}\,dt
    \end{align*}
    of a signal $\psi$ with a kernel $\mu$ and then synthesize a new function from the result
    \begin{align}\label{eq:main_def}
        B_\mu^\varphi \psi(t) = V_\varphi^* (\mu * V_\varphi \psi)(t) = \int_{\R^{2d}} \mu * V_\varphi \psi(x,\omega) \varphi(t-x) e^{2\pi i \omega \cdot t} \,dx\,d\omega.
    \end{align}
    We refer to $B_\mu^\varphi$ as a \emph{time-frequency blurring operator} or \emph{STFT convolver} with window $\varphi$ and kernel~$\mu$.

    \begin{figure}[H]
        \centering
        \includegraphics[width=0.92\linewidth]{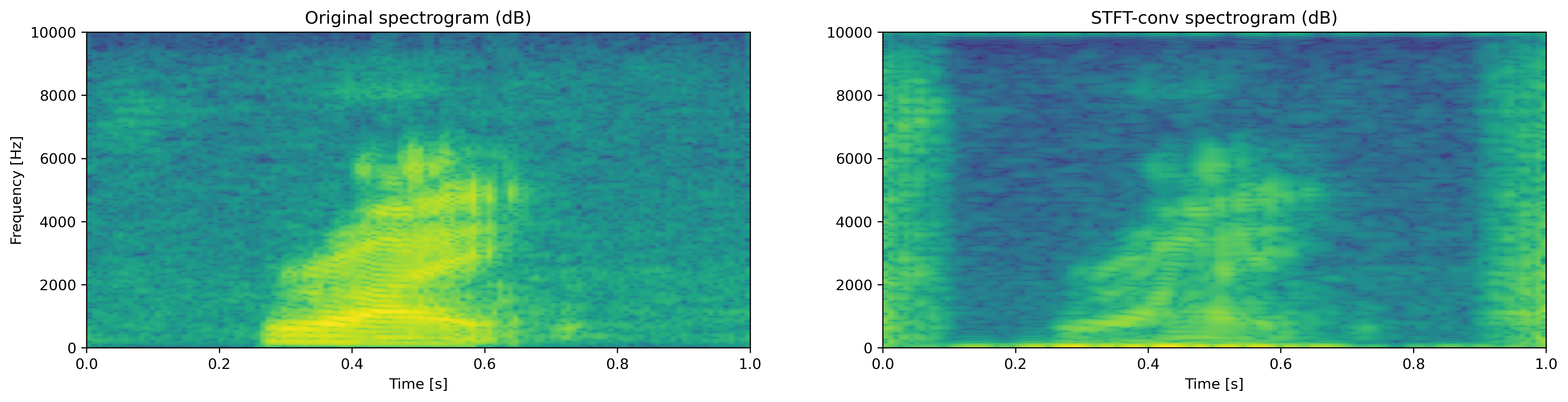}
        \caption{Action of a time-frequency blurring operator with Gaussian kernel on an audio recording, illustrated using spectrograms.}
        \label{fig:non-mel-blur-op}
    \end{figure}

    A similar operator was investigated as part of earlier work by Dörfler and Torrésani \cite{Dorfler2007, Dorfler2009} where a twisted convolution was instead used to represent general Hilbert-Schmidt operators in the time-frequency domain. As we will see, the change to a regular convolution greatly affects the properties of the operator.

    Our main motivation for studying the operator \eqref{eq:main_def} is its use as a data augmentation method for signals in machine learning. This area has seen much interest in recent years with applications in several areas of signal processing such as automatic speed recognition \cite{Park2019, Wang2019}, EEG analysis \cite{Rommel2022}, keyword spotting \cite{raju2018data} and sound classification \cite{AbayomiAlli2022}, and several augmentation libraries released \cite{papakipos2022augly, McFee2015, audiomentations, wavaugment2020, sigment, nlpaug}. With the rapid rise in popularity of vision transformer (ViT) based methods for signal analysis \cite{dosovitskiy2021, Gong2021, Dong2018}, which require more training data compared to convolutional methods \cite{dosovitskiy2021}, there is an increased need for robust augmentation methods. Standard techniques in this field include but are not limited to time-frequency masking, adding noise, preemphasis, simulating room impulse response and changing volume, pitch and speed.
    
    Augmentation methods which act directly on the spectrogram, the squared modulus of the STFT, have recently received a great deal of attention \cite{Park2019, Wang2021spec, Wei2020, saber} and have been used to achieve state of the art results in a variety of tasks \cite{Zhang2020, Chen2022, Vygon2021, Radford2022}. These methods can be easy to implement in the training pipeline as neural networks often use the spectrogram as an input feature. An obvious drawback is that the resulting augmented spectrograms do not necessarily correspond to spectrograms of actual signals which possess certain smoothness properties. Consequently, the training data may be out of distribution with respect to actual real world data. Should one wish to construct the closest corresponding waveform of the augmented spectrogram, a costly and inexact reconstruction method such as the Griffin–Lim algorithm \cite{Griffin1984} must be used.
    
    The proposed blurring operator instead acts on the STFT of the signal which allows for computationally efficient waveform synthesis since phase information is not lost as is the case for spectrogram augmentations. Upon computing the STFT, we obtain an in-distribution spectrogram. Many of the standard methods for signal data augmentation can be realized as STFT multipliers with different symbols but have traditionally been implemented in the waveform domain instead, c.f. \cite{audiomentations, wavaugment2020}. In \cite{Balazs2024}, the quantitative differences between these approaches was investigated. It is our belief that phase space based augmentation methods are a valuable tool and from this perspective, the introduction of a STFT convolver is a natural next step after STFT multipliers whose associated methods have seen widespread use and success. 

    The reason we expect our operator to be suitable for signal data augmentation is that it preserves high level phase space features while only changing the local structure. By this we mean that features such as a clear base note with regularly spaced harmonics or a chirp are preserved while the high resolution patterns in phase space are altered significantly. While these patterns certainly define an important part of the characteristic of the signal, the higher level structure stays intact.

    The applied reader may focus their attention on Sections \ref{sec:related_objects} and \ref{sec:numerical_implementation} for a deeper discussion on implementation considerations for which much of the preliminaries in Section \ref{sec:preliminaries} are not required.  

    \subsubsection*{Notational conventions}
    For functions on $\R^{2d}$, we will use $L^{p,q}(\R^{2d})$ to denote the mixed-norm Lebesgue spaces with $p$-norm in the first $d$ variables and $q$-norm in the remaining $d$ variables. The Fourier transform of a function $f$ will be denoted by $\hat{f}$ and use the standard normalization $\hat{f}(\omega) = \int_{\R^d} f(t) e^{-2 \pi i \omega \cdot t}\,dt$. By $\mathcal{F} L^p(\R^d)$ we will mean the Fourier-Lebesgue space consisting of functions whose Fourier transforms is in $L^p(\R^d)$. A check over a function, $\check{\psi}$, will mean that the argument is negated i.e., $\check{\psi}(t) = \psi(-t)$. We will take our ambient space to be $L^2(\R^d)$ and so norms $\Vert \cdot \Vert$ and inner products $\langle \cdot, \cdot \rangle$ without subscripts are to be understood to be taken in this space. The Schwartz space on $\R^d$ will be denoted by $\mathcal{S}(\R^d)$.

    \section{Time-frequency preliminaries}\label{sec:preliminaries}
    In this section we go over some of the key objects from time-frequency analysis which we will make use of below. For a more comprehensive overview, the reader is referred to \cite{grochenig_book, Wong2002, Wong1998, CORDERO2007, Suppappola2018}.

    \subsection{Short-time Fourier transform}  
    For a signal $\psi \in L^2(\R^d)$ and a window $\varphi \in L^2(\R^d)$, the short-time Fourier transform of $\psi$ with respect to $\varphi$ is defined as 
    \begin{align}\label{eq:def_stft}
        V_\varphi \psi(x, \omega) = \langle \psi, \pi(x, \omega)\varphi \rangle = \int_{\R^d} \psi(t) \overline{\varphi(t-x)}e^{-2\pi i \omega \cdot t}\,dt
    \end{align}
    where $\pi(x, \omega)$ is a time-frequency shift by time $x$ and frequency $\omega$, defined as $\pi(x, \omega) = M_\omega T_x$ where $M_\omega f(t) = e^{2\pi i \omega \cdot t}$ is the modulation operator and $T_xf(t) = f(t-x)$ is the translation operator. We often write $z = (x,\omega) \in \R^{2d}$ for the coordinates of $x, \omega$ in time-frequency space $\R^{2d}$. As is indicated in the middle step of the definition \eqref{eq:def_stft}, this transform projects the signal onto time-frequency shifted versions of the window and the interpretation is that $V_\varphi \psi(x, \omega)$ measures the importance of the time $t$ and frequency $\omega$ to $\psi$. 

    One of the key properties of the STFT is \emph{Moyal's identity} which states that
    \begin{align}\label{eq:moyal}
        \big\langle V_{\varphi_1}\psi_1, V_{\varphi_2}\psi_2 \big\rangle_{L^2(\R^{2d})} = \langle \psi_1, \psi_2 \rangle \overline{\langle \varphi_1, \varphi_2 \rangle}.
    \end{align}
    As a consequence, for a normalized window $\varphi$, $V_\varphi : L^2(\R^d) \to L^2(\R^{2d})$ is an isometry and the STFT mapping can be inverted in the sense that
    \begin{align}\label{eq:moyal_reconstruction}
        \psi = \frac{1}{\langle \varphi_1, \varphi_2 \rangle} V_{\varphi_2}^* (V_{\varphi_1} \psi) = \frac{1}{\langle \varphi_1, \varphi_2 \rangle} \int_{\R^{2d}} V_{\varphi_1} \psi(z) \pi(z) \varphi_2 \,dz
    \end{align}
    weakly.

    The STFT is in general complex valued and in many cases the phase information is superfluous. For this reason, the squared modulus $|V_\varphi \psi|^2$, which is called the spectrogram, is often used in applications. The process of recovering the STFT from the spectrogram, meaning to invert the mapping $V_\varphi \psi \mapsto |V_\varphi \psi|^2$, is called \emph{phase retrieval} and is an area of rich research.

    \subsection{Gabor spaces}
    The image of the short-time Fourier transform, $V_\varphi(L^2(\R^d)) \subset L^2(\R^{2d})$, is called the \emph{Gabor space} associated to the window $\varphi$. We saw earlier in \eqref{eq:moyal_reconstruction} how the adjoint of the STFT mapping $V_\varphi$ can be used to synthesize a signal from a function on $\R^{2d}$. If that original function is an STFT, the synthesized signal will be precisely the original signal used to compute the STFT. However, if we take a function in $L^2(\R^{2d})$, synthesize a signal from it with $V_\varphi^*$ and then compute the STFT, this is precisely orthogonally projecting the original function onto the Gabor space. We can write both of these facts as 
    \begin{align}\label{eq:gabor_proj}
        V_\varphi^* V_\varphi = I_{L^2}, \qquad V_\varphi V_\varphi^* = P_{V_\varphi(L^2)}
    \end{align}
    where $P_{V_\varphi(L^2)}$ is the orthogonal projection onto the Gabor space $V_\varphi(L^2(\R^d))$.

    \subsection{Localization operators}
    By introducing a multiplication operator in between $V_\varphi$ and $V_\varphi^*$ in $V_\varphi^* V_\varphi = I_{L^2}$, we can effectively choose which regions of the time-frequency plane we want to prioritize when reconstructing a signal. The resulting operator, called a localization operator or STFT multiplier, was first investigated by I. Daubechies in \cite{daubechies1988_loc} and has since seen wide applications in signal analysis, pseudo-differential operators and partial differential equations. We write it formally as
    \begin{align}\label{eq:loc_op}
        A_m^\varphi \psi = V_\varphi^*(m \cdot V_\varphi \psi) = \int_{\R^{2d}} m(z) V_\varphi \psi(z) \pi(z)\varphi\,dz 
    \end{align}
    and refer to the function $m : \R^{2d} \to \C$ as the \emph{mask} or \emph{symbol} of the operator. The mask is often taken to be an indicator function of some subset of the time-frequency plane. In this way, we can restrict a signal to this subset when reconstructing it.

    At times, we separate the analysis window in $V_\varphi$ and the synthesis window in $V_\varphi^*$ and use the relation $V_{\varphi_2}^* V_{\varphi_1} = \langle \varphi_1, \varphi_2 \rangle I_{L^2}$ from \eqref{eq:moyal_reconstruction}, leading to the more general operator
    \begin{align}\label{eq:loc_op_general}
        A_m^{\varphi_1, \varphi_2} \psi = V_{\varphi_2}^*(m \cdot V_{\varphi_1} \psi) = \int_{\R^{2d}} m(z) V_{\varphi_1} \psi(z) \pi(z)\varphi_2 \,dz. 
    \end{align}

    \subsection{Modulation spaces}
    Often when dealing with mapping properties of localization operators and other problems in time-frequency analysis, modulation spaces turn out to be a suitable setting. They are a class of function spaces first introduced by Feichtinger in \cite{FEICHTINGER1989307} which are characterized by the integrability properties of short-time Fourier transforms. For a non-zero Schwartz window $\varphi \in \mathcal{S}(\R^d)$, we can define the modulation space $M^{p,q}(\R^d)$ for $1 \leq p, q \leq \infty$ as the set of all tempered distributions $\psi$ such that
    \begin{align*}
        \Vert \psi \Vert_{M^{p,q}(\R^d)} = \left( \int_{\R^d} \left( \int_{\R^d} |V_\varphi \psi(x,\omega)|^p \,dx \right)^{q/p} d\omega \right)^{1/q} < \infty.
    \end{align*}
    When $p$ or $q$ are equal to $\infty$, we make suitable adjustments and for $p=q$ we use the shorthand $M^{p,p}(\R^d) = M^p(\R^d)$. The norms induced by different Schwartz windows are all equivalent and so we do not indicate the window when denoting a modulation space. From \eqref{eq:moyal}, it is clear that $M^2(\R^d) = L^2(\R^d)$. The case of $p = 1$ is of special importance and the space $M^1(\R^d)$ is often called \emph{Feichtinger's algebra}.

    We can interpolate between different modulation spaces in the same way as we do for Lebesgue spaces as described in the following simplified lemma which follows from \cite[Theorem 1.1]{Guo2015}.
    \begin{lemma}\label{lemma:interpolate_modulation}
        Let $1 \leq p_1, p_2 \leq \infty$ and $p_\theta$ be defined by the relation
        $$
        \frac{1}{p_\theta} = \frac{1-\theta}{p_1} + \frac{\theta}{p_2}.
        $$
        Then we can interpolate between $M^{p_1}(\R^d)$ and $M^{p_2}(\R^d)$ as
        \begin{align*}
            \Big[ M^{p_1}(\R^d),\, M^{p_2}(\R^d) \Big]_\theta = M^{p_\theta}(\R^d).
        \end{align*}
    \end{lemma}
    Later we will also have use for the following mapping property which is a special case of \cite[Proposition 11.3.7]{grochenig_book}.
    \begin{lemma}\label{lemma:stft_adjoint_mapping}
        If $\varphi \in M^1(\R^d)$, then $V_\varphi^*$ maps $L^{p,q}(\R^{2d})$ to $M^{p,q}(\R^d)$ for $1 \leq p,q \leq \infty$ with
        \begin{align*}
            \Vert V_\varphi^* F \Vert_{M^{p,q}(\R^d)} \lesssim \Vert \varphi \Vert_{M^1(\R^d)} \Vert F \Vert_{L^{p,q}(\R^{2d})}
        \end{align*}
        and $\psi \mapsto \Vert V_\varphi \psi \Vert_{L^{p,q}(\R^{2d})}$ is an equivalent norm on $M^{p,q}(\R^d)$.
    \end{lemma}

    \subsection{Quadratic time-frequency distributions}
    The spectrogram $|V_\varphi \psi|^2$ is just one example of a time-frequency distribution which has a quadratic dependence on $\psi$. Another widely used object is the Wigner distribution
    \begin{align*}
        W(\psi)(x,\omega) = \int_{\R^d} \psi\left(x+\frac{t}{2}\right)\overline{\psi\left(x-\frac{t}{2}\right)} e^{-2\pi i \omega \cdot t}\,dt
    \end{align*}
    defined for $\psi \in L^2(\R^d)$. More generally, one can characterize a class of well-behaved quadratic time-frequency representations as those of the form
    $$
    Q_\Phi(\psi) = W(\psi) * \Phi
    $$
    for some tempered distribution $\Phi$. These time-frequency distributions are said to belong to \emph{Cohen's class of quadratic time-frequency distributions}.

    \section{Related objects and adaptations}\label{sec:related_objects}    
    For applications, the formulation \eqref{eq:main_def} may not always be the optimal incarnation of the idea of blurring a signal in time and frequency. In this section, we motivate this statement by detailing alternatives which may be better suited to certain applications.
    
    \subsection{Spectrogram blurring}\label{sec:spectrogram_blurring}
    As discussed in the introduction, many contemporary signal augmentation methods act directly on the spectrogram. In particular, the widely popular SpecAugment augmentation method \cite{Park2019} multiplies the spectrogram by a time-frequency mask. Here there is a clear analog in localization operators and we should expect that differences between SpecAugment and applying a localization operator with the same mask should only be visible near the edges of the mask. For time-frequency blurring operators however, performing the blurring action on the STFT instead of the spectrogram leads to a noticeable difference as the phase data of the STFT greatly affects the convolution. For example, the STFT of white noise is largely unstructured while sinusoids and impulses give rise to predictable phase gradients, see e.g \cite[Figure 1]{Prusa2017} and \cite[Figure 2]{Giorgi2022}, and hence a blurring operation which takes phase into account should effect white noise differently than such a signal.
    
    Performing the convolution on the spectrogram is easier to implement and the intuition behind ``spreading'' the time-frequency contents of the signal stays intact. To the best of our knowledge, this approach has previously only been mentioned in the GitHub repository \cite{saber} but never studied in any detail.

    Analyzing spectrogram blurring as a mapping on $L^2(\R^d)$ is in general intractable as it involves a phase retrieval step to obtain a waveform. By instead looking at spectrogram blurring as a mapping from $L^2(\R^d)$ to $L^1(\R^{2d})$, we see that blurring a spectrogram exactly corresponds to switching quadratic time-frequency distribution from the spectrogram to a Cohen's class distribution \cite{grochenig_book, Hlawatsch1992, Suppappola2018}. Indeed,
    \begin{align*}
        \mu * |V_\varphi \psi|^2 = \mu * \big(W(\psi) * W(\check{\varphi})\big) = W(\psi) * \big( \mu * W(\check{\varphi}) \big)
    \end{align*}
    and so a blurred spectrogram is a well behaved quadratic time-frequency distribution.

    In Section \ref{sec:examples}, we present a brief numerical comparison between spectrogram blurring and the time-frequency blurring operator where we refer to the spectrogram blurring procedure as SpecBlur.
     
    \subsection{Position-dependent kernel}\label{sec:pos_dep_ker}
    An obvious generalization of $B_\mu^\varphi$ is to allow the kernel $\mu$ to depend on $z \in \R^{2d}$. We can write this as
    \begin{align*}
        B_{\boldsymbol{\mu}}^\varphi \psi = \int_{\R^{2d}} \mu_z * V_\varphi \psi(z) \pi(z)\varphi \,dz
    \end{align*}
    where we use a bold $\boldsymbol{\mu}$ to indicate that it is a function on the double phase space $\R^{4d}$ with $\boldsymbol{\mu}(z,w) = \mu_z(w)$. This is a very general form of operator and e.g. localization operators can be realized as a special case by
    \begin{align}\label{eq:blur_gen_loc}
        \mu_z = m(z)\delta_0 \implies B_{\boldsymbol{\mu}}^\varphi = A_m^\varphi    
    \end{align}
    since $B_{\delta_0}^\varphi = V_\varphi^* V_\varphi = I_{L^2}$. In view of this, these generalized time-frequency blurring operators may provide a nice model for approximating operators or interpolating between them, see e.g. \cite{Olivero2012} for work in a similar direction for localization operators. Indeed, such an object has already been used in the discrete setting in \cite{Mack2020, Schroter2022} for noise reduction under the name ``Deep Filtering''. In \cite{Mack2020, Schroter2022} and related works, the STFT and engineered additional features are used as the input to a neural network which outputs the position dependent kernels $\mu_z$ with the goal of minimizing some distance function
    \begin{align*}
        d\big(B_{\boldsymbol{\mu}}^\varphi \psi_{\text{noisy}}, \psi_{\text{clean}} \big).
    \end{align*}
    Earlier work on noise reduction, e.g., \cite{Williamson_2016}, has focused on learning an optimal time-frequency mask which, in view of \eqref{eq:blur_gen_loc}, is just a special case of a time-frequency blurring operator. On a high level, the deep filtering papers have exchanged a function on $\R^{2d}$ ($m$) for one on $\R^{4d}$ $(\boldsymbol{\mu})$ or, equivalently, imposed structure on the final layer of their neural network.
    
    We will not investigate the properties of these more heavily parameterized operators further but note that the present paper is likely to be a good first step in this direction.

    \subsection{Window generalizations}\label{sec:generalized_windows}
    We briefly discuss two generalizations of time-frequency blurring operators which both have clear counterparts for localization operators. The first is separating the analysis and synthesis windows as was done for localization operators in \eqref{eq:loc_op_general}. This leads to the following operator 
    \begin{align*}
        B_\mu^{\varphi_1, \varphi_2} \psi = V_{\varphi_2}^* (\mu * V_{\varphi_1} \psi) = \int_{\R^{2d}} \mu * V_{\varphi_1} \psi(z)\,\pi(z)\varphi_2\,dz.
    \end{align*}
    Many of the results which will be derived in Section \ref{sec:analytic_properties} can be adapted to this generalized operator. The required changes to the proofs are minimal and routine and, in the interest of brevity, we leave them to the interested reader.

    The concept of multi-window localization operators, defined as linear combinations of localization operators, $\sum_{n=1}^N A_m^{\varphi_1^n, \varphi_2^n}$, was further generalized to permit \emph{operator windows} $S$ of trace-class in \cite{Luef2019} as a part of the framework of quantum harmonic analysis \cite{Werner1984}. Without going into too much detail, we note that these are defined via the Bochner integral
    \begin{align*}
        A_m^S = \int_{\R^{2d}} m(z) \pi(z) S \pi(z)^* \,dz
    \end{align*}
    which reduces down to multi-window localization operator when the operator $S$ is of finite rank. This formulation of localization operators has proven useful \cite{Luef2019, Halvdansson2023, Luef2021, Luef2019_acc}. In our case, the generalization to operator windows $S$ of trace-class can be written as
    \begin{align*}
        B_\mu^S \psi = \int_{\R^{2d}} \mu * \big(\pi(z) S \pi(\cdot)^* \psi \big)(z)\,dz.
    \end{align*}
    Note specifically that if $S$ has the singular value decomposition $S = \sum_n s_n (\varphi_n^1 \otimes \varphi_n^2)$, it can be written as
    \begin{align*}
        B_\mu^S = \sum_n s_n B_\mu^{\varphi_n^1, \varphi_n^2},
    \end{align*}
    i.e., a linear combination of time-frequency blurring operators.

    \section{Analytical properties}\label{sec:analytic_properties}

    \subsection{Mapping properties}\label{sec:mapping_props}
    Up until now we have left the spaces which $\varphi$ and $\mu$ belong to ambiguous. We will see that the most natural assumptions are that the kernel $\mu$ is a bounded measure and that the window $\varphi$ is square integrable or in Feichtinger's algebra, similar to the convention for localization operators. Changing these assumptions affects which spaces the operator maps to and from. While making no claims of completeness, we collect some of the more easily available results on these mapping properties in this subsection.

    Our first result provides the most natural boundedness condition on $L^2(\R^d)$.
    \begin{proposition}\label{prop:bounded_op}
        Let $\mu \in M(\R^{2d})$ and $\varphi \in L^2(\R^d)$, then $B_\mu^\varphi$ is a bounded operator on $L^2(\R^d)$ with
        \begin{align*}
            \Vert B_\mu^\varphi \Vert_{B(L^2(\R^d))} \leq \Vert \mu \Vert_{M(\R^{2d})} \Vert \varphi \Vert^2.
        \end{align*}
    \end{proposition}
    \begin{proof}
        For $\psi \in L^2(\R^d)$ we compute
        \begin{align*}
            \Vert B_\mu^\varphi \psi \Vert &=  \frac{1}{\Vert \varphi \Vert} \big\Vert V_\varphi( V_\varphi^*( \mu * V_\varphi) ) \big\Vert_{L^2(\R^{2d})}\\
            &\leq \Vert \varphi \Vert\Vert \mu * V_\varphi \psi \Vert_{L^2(\R^{2d})}\\
            &\leq \Vert \mu \Vert_{M(\R^{2d})} \Vert \varphi \Vert \Vert V_\varphi\psi \Vert_{L^2(\R^{2d})}\\
            &=\Vert \mu \Vert_{M(\R^{2d})} \Vert \varphi \Vert^2 \Vert \psi \Vert
        \end{align*}
        where we used that $V_\varphi V_\varphi^* = P_{V_\varphi(L^2)}$ is an orthogonal projection, Young's inequality and Moyal's identity \eqref{eq:moyal} twice.
    \end{proof}
    Next we look at mapping between modulation spaces for general exponents.
    \begin{proposition}\label{prop:mapping_master}
        Let $\mu \in L^{p_1, p_2}(\R^{2d})$, $\psi \in M^{q_1, q_2}(\R^d)$ and $\varphi \in M^1(\R^d)$ with $\frac{1}{p_i} + \frac{1}{q_i} = 1 + \frac{1}{r_i}$ and $1 \leq p_i, q_i, r_i \leq \infty$ for $i = 1,2$. Then $B_\mu^\varphi \psi \in M^{r_1, r_2}(\R^d)$ with
        $$
        \Vert B_\mu^\varphi \psi\Vert_{M^{r_1, r_2}(\R^d)} \lesssim \Vert \mu \Vert_{L^{p_1, p_2}(\R^{2d})}\Vert \psi \Vert_{M^{q_1, q_2}(\R^{d})}
        $$
        where the implicit constant depends on $\varphi$ and the $p, q, r$ constants.
    \end{proposition}
    \begin{proof}
        For $\psi \in M^{q_1, q_2}(\R^d)$, we have that
        \begin{align*}
            \Vert B_\mu^\varphi \psi \Vert_{M^{r_1, r_2}(\R^d)} &= \Vert V_\varphi^*( \mu * V_\varphi \psi ) \Vert_{M^{r_1, r_2}(\R^d)}\\
            &\lesssim \Vert \mu * V_\varphi \psi \Vert_{L^{r_1, r_2}(\R^{2d})} && \big(\text{Lemma \ref{lemma:stft_adjoint_mapping}}\big)\\
            &\leq \Vert \mu \Vert_{L^{p_1, p_2}(\R^{2d})} \Vert V_\varphi\psi \Vert_{L^{q_1, q_2}(\R^{2d})} && \big( \text{Young's inequality} \big)\\
            &\lesssim \Vert \mu \Vert_{L^{p_1, p_2}(\R^{2d})} \Vert \psi \Vert_{M^{q_1, q_2}(\R^{d})}
        \end{align*}
        where we used the equivalence of norms for modulation spaces in the last step.
    \end{proof}
    A proof of the general version of Young's inequality which we used above can be found in e.g. \cite[Proposition 11.1.3]{grochenig_book}.

    If we want the output $B_\mu^\varphi \psi$ to be in an $L^p$ space instead of a modulation space, we can get this out of the above result with some minor work.
    \begin{proposition}\label{prop:interpol_Mp_Lp_mapping}
        Fix $1 \leq p \leq 2$ and let $\varphi \in M^1(\R^d)$, $\mu \in L^1(\R^{2d})$ and $\psi \in M^p(\R^d)$, then
        \begin{align*}
            \Vert B_\mu^\varphi \psi \Vert_{L^p(\R^d)} \lesssim \Vert \mu \Vert_{L^1(\R^{2d})} \Vert \psi \Vert_{M^p(\R^d)}
        \end{align*}
        where the implicit constant is dependent on $\varphi$ and $p$.
    \end{proposition}
    \begin{proof}
        We wish to interpolate between the $p=1$ and $p=2$ cases of this. Since $M^2(\R^d) = L^2(\R^d)$ as noted earlier, the $p=2$ version is just Proposition \ref{prop:bounded_op}. Meanwhile for $p=1$ it suffices to note that $L^1(\R^d)$ can be embedded in $M^1(\R^d)$ by \cite[Proposition 12.1.4]{grochenig_book}. Now an application of Lemma \ref{lemma:interpolate_modulation} yields the desired result.
    \end{proof}    
    If we want $B_\mu^\varphi \psi$ to belong to $L^\infty(\R^d)$, we can place additional assumptions on the window function.
    \begin{proposition}\label{prop:Linfty_cond}
        Let $\varphi \in L^2(\R^d) \cap L^\infty(\R^d)$, $\mu \in L^1(\R^{2d})$ and $\psi \in M^1(\R^d)$, then
        \begin{align*}
            \Vert B_\mu^\varphi \psi \Vert_{L^\infty(\R^d)} \lesssim \Vert \mu \Vert_{L^1(\R^{2d})}  \Vert \varphi \Vert_{L^\infty(\R^d)} \Vert \psi \Vert_{M^1(\R^d)}.
        \end{align*}
    \end{proposition}
    \begin{proof}
        The simplest way to show this is to consider the dual formulation of the $L^\infty(\R^d)$ norm. Indeed,
        \begin{align*}
            \Vert B_\mu^\varphi \psi \Vert_{L^\infty(\R^d)} &= \sup_{\Vert \phi \Vert_{L^1} = 1}  \left| \left\langle \int_{\R^{2d}} \mu * V_\varphi\psi(z) \pi(z)\varphi\,dz, \phi \right\rangle  \right|\\
            &= \sup_{\Vert \phi \Vert_{L^1} = 1}  \left| \int_{\R^{2d}} \mu * V_\varphi\psi(z) \langle \pi(z)\varphi, \phi\rangle\,dz \right|\\
            &\leq \sup_{\Vert \phi \Vert_{L^1} = 1} \left\Vert \mu * V_\varphi \psi(\cdot) \langle \pi(\cdot) \varphi, \phi \rangle \right\Vert_{L^1(\R^{2d})}\\
            &\leq \sup_{\Vert \phi \Vert_{L^1} = 1} \left\Vert \mu * V_\varphi \psi \right\Vert_{L^1(\R^{2d})} \left\Vert \langle \pi(\cdot) \varphi, \phi \rangle \right\Vert_{L^\infty(\R^{2d})}\\
            &\lesssim \Vert \mu \Vert_{L^1(\R^{2d})} \Vert \psi \Vert_{M^1(\R^d)} \Vert \varphi \Vert_{L^\infty(\R^d)}.
        \end{align*}
    \end{proof}
    In particular, by combining Proposition \ref{prop:interpol_Mp_Lp_mapping} and Proposition \ref{prop:Linfty_cond} we see that when $\mu \in L^1(\R^{2d}), \varphi \in M^1(\R^d) \subset L^2(\R^d) \cap L^\infty(\R^d)$ (this follows from \cite[Proposition 12.1.4]{grochenig_book}) and $\psi \in M^1(\R^d)$, $B_\mu^\varphi \psi$ is in $L^1(\R^d) \cap L^\infty(\R^d)$.

    Lastly we show that if the window and kernel are in the Schwartz space, $B_\mu^\varphi$ maps the Schwartz space to itself. To do so, we will need two preliminary results.
    \begin{lemma}[{\cite[Proposition 11.2.4]{grochenig_book}}]\label{lemma:adjoint_stft_schwartz}
        Fix $\varphi \in \mathcal{S}(\R^d)$ and assume $F : \R^{2d} \to \C$ has rapid decay on $\R^{2d}$, then the integral
        $$
        \psi(t) = \int_{\R^{2d}} F(z) \pi(z)\varphi(t)\,dz
        $$
        defines a function $\psi \in \mathcal{S}(\R^d)$.
    \end{lemma}
    In particular, the above lemma implies that if $F \in \mathcal{S}(\R^{2d})$, then $V_\varphi^* F \in \mathcal{S}(\R^d)$.
    \begin{lemma}[{\cite[Theorem 11.2.5]{grochenig_book}}]\label{lemma:schwartz_stft_keeps}
        Fix $\varphi \in \mathcal{S}(\R^d)$, then $\psi$ is in the Schwartz space if and only if $V_\varphi \psi \in \mathcal{S}(\R^{2d})$.
    \end{lemma}
    We can now proceed with the main proposition.
    \begin{proposition}
        Let $\mu \in \mathcal{S}(\R^{2d})$ and $\varphi \in \mathcal{S}(\R^d)$, then
        $$
        B_\mu^\varphi : \mathcal{S}(\R^d) \to \mathcal{S}(\R^d).
        $$
    \end{proposition}
    \begin{proof}
        Let $\psi \in \mathcal{S}(\R^d)$, then by Lemma \ref{lemma:schwartz_stft_keeps} $V_\varphi \psi \in \mathcal{S}(\R^d)$. Now $B_\mu^\varphi \psi = V_\varphi^*(\mu * V_\varphi \psi)$ and so by Lemma \ref{lemma:adjoint_stft_schwartz}, the result follows upon noting that the convolution $\mu * V_\varphi \psi$ between two Schwartz functions is another Schwartz function.
    \end{proof}
    
    \subsection{Weak action and positivity}\label{section:weak_action}
    In the case of localization operators, computing the weak action has proven valuable \cite{Cordero2003, Luef2018}. Inspired by this, we look at the weak action of $B_\mu^\varphi$ in this section. By moving the adjoint of $V_\varphi$ to the other side of the inner product, we immediately get
    \begin{align*}
        \big\langle B_\mu^\varphi \psi, \phi \big\rangle = \big\langle \mu * V_\varphi \psi, V_\varphi \phi \big\rangle_{L^2(\R^{2d})}.
    \end{align*}
    Of course, we cannot move over the convolution but in an effort to obtain something more workable, we take the Fourier transform of both sides and rearrange, yielding
    \begin{align}\label{eq:weak_action_parseval}
        \big\langle B_\mu^\varphi \psi, \phi \big\rangle = \Big\langle \hat{\mu} \cdot \widehat{V_\varphi \psi}, \widehat{V_\varphi \phi} \Big\rangle_{L^2(\R^{2d})} = \Big\langle \hat{\mu} , \overline{\widehat{V_\varphi \psi}} \cdot \widehat{V_\varphi \phi} \Big\rangle_{L^2(\R^{2d})}.
    \end{align}
    From this formulation we can obtain a general bound on $\langle B_\mu^\varphi \psi, \phi \rangle$. Before the proof, we recall the Haussdorff-Young inequality which states that if $\frac{1}{p} + \frac{1}{q} = 1$ with $p \in [1,2]$ and $F \in L^p$, then
    \begin{align*}
         \Vert \hat{F} \Vert_{L^q} \leq \Vert F \Vert_{L^p}.
    \end{align*}
    \begin{proposition}
        Let $1 \leq p \leq \infty$ and $1 \leq q, r \leq 2 \leq q', r' \leq \infty$ be such that $\frac{1}{p} + \frac{1}{q'} + \frac{1}{r'} = 1$, $\frac{1}{q} + \frac{1}{q'} = 1$ and $\frac{1}{r} + \frac{1}{r'} = 1$. If $\mu \in \mathcal{F} L^p(\R^{2d})$, $\psi\in M^{q}(\R^d)$, $\phi \in M^{r}(\R^d)$ and $\varphi \in M^1(\R^d)$, then
        \begin{align*}
            \big|\big\langle B_\mu^\varphi \psi, \phi \big\rangle\big| \lesssim \Vert \hat{\mu} \Vert_{L^p(\R^{2d})} \Vert \psi \Vert_{M^{q}(\R^d)} \Vert \phi \Vert_{M^{r}(\R^d)}.
        \end{align*}
    \end{proposition}
    \begin{proof}
        From \eqref{eq:weak_action_parseval}, we have that
        $$
        \big|\big\langle B_\mu^\varphi \psi, \phi \big\rangle\big| \leq \big\Vert \hat{\mu} \cdot \widehat{V_\varphi \psi} \cdot \widehat{V_\varphi \phi} \big\Vert_{L^1(\R^{2d})}
        $$
        on which we can apply generalized Hölder. This yields
        \begin{align*}
            \big|\big\langle B_\mu^\varphi \psi, \phi \big\rangle\big| \leq \big\Vert \hat{\mu} \big\Vert_{L^p(\R^{2d})} \big\Vert \widehat{V_\varphi \psi} \big\Vert_{L^{q'}(\R^{2d})} \big\Vert \widehat{V_\varphi \phi} \big\Vert_{L^{r'}(\R^{2d})}.
        \end{align*}
        Upon applying the Haussdorff-Young inequality followed by Lemma \ref{lemma:stft_adjoint_mapping} to the two last quantities, the desired result follows.
    \end{proof}

    From \eqref{eq:weak_action_parseval}, we can also get a condition for the positivity of $B_\mu^\varphi$.
    \begin{proposition}
        Let $\varphi \in L^2(\R^d)$ and $\mu \in M(\R^{2d})$ with $\hat{\mu} \geq 0$. Then $B_\mu^\varphi$ is a positive operator on $L^2(\R^d)$.
    \end{proposition}
    \begin{proof}
        If $\hat{\mu} \geq 0$, non-negativity of $\langle B_\mu^\varphi \psi, \psi \rangle$ clearly follows from \eqref{eq:weak_action_parseval} as
        \begin{align*}
            \big\langle B_\mu^\varphi \psi, \psi \big\rangle = \big\langle \hat{\mu}, \big|\widehat{V_\varphi \psi}\big|^2 \big\rangle_{L^2(\R^{2d})}
        \end{align*}
        and finiteness follows from that $B_\mu^\varphi$ is bounded on $L^2(\R^d)$ by Proposition \ref{prop:mapping_master}.
    \end{proof}
    The next proposition essentially states that there exists nontrivial kernels such that the convolution operator is the zero operator on the Gabor space $V_\varphi(L^2(\R^d))$, which is not the case for $L^2(\R^{2d})$.
    \begin{proposition}
        There exists non-zero $\varphi \in L^2(\R^d)$ and $\mu \in L^1(\R^{2d})$ such that $B_\mu^\varphi$ is the zero operator.
    \end{proposition}
    \begin{proof}
        We will tacitly choose $\mu, \varphi$ such that the integrand in the integral defining the inner product \eqref{eq:weak_action_parseval} is always zero so that $\langle B_\mu^\varphi \psi, \phi \rangle = 0$ for all $\psi, \phi \in L^2(\R^d)$. First, we compute the Fourier transform of a general STFT $V_\varphi \psi$. The identity
        \begin{align*}
            V_\varphi\psi(x, \omega) = M_{-\omega} (\psi*M_\omega \varphi^*)(x),
        \end{align*}
        where $\varphi^*(x) = \overline{\varphi(-x)}$, from \cite[Lemma 3.1.1]{grochenig_book} simplifies this considerably. Let $\mathcal{F}_1$ denote the Fourier transform in the first $d$ variables and $\mathcal{F}_2$ the Fourier transform in the last $d$ variables, we then have that
        \begin{align*}
            \mathcal{F}_1(V_\varphi\psi)(x', \omega) &= \mathcal{F}_1 \big(M_{-\omega}(\psi * M_\omega \varphi^*)\big)(x')\\
            &= T_{-\omega} \big(\hat{\psi} \cdot\widehat{M_\omega \varphi^*}\big)(x')\\
            &=\hat{\psi}(x' + \omega) \widehat{M_\omega \varphi^*}(x'+\omega)\\
            &=\hat{\psi}(x' + \omega) \widehat{\varphi^*}(x').
        \end{align*}
        The full Fourier transform can now be computed as
        \begin{align*}
            \widehat{V_\varphi \psi}(x', \omega') &= \mathcal{F}_2\big(\hat{\psi}(x' + \cdot)\widehat{\varphi^*}(x')\big)(\omega')\\
            &=\check{\psi}(\omega') e^{2\pi i x' \cdot \omega'} \widehat{\varphi^*}(x').
        \end{align*}
        From this we see that the signal $\psi$ has no effect on the support in the first $d$ variables of $\widehat{V_\varphi \psi}$. If we now choose $\varphi$ such that $\widehat{\varphi^*}$ is zero on some set $E \subset \R^d$ and $\mu$ such that the first $d$ variables of $\hat{\mu}$ is supported on the same set $E$, it will hold that
        \begin{align*}
            \big\langle B_\mu^\varphi \psi, \phi \big\rangle &= \Big\langle \hat{\mu} , \overline{\widehat{V_\varphi \psi}} \cdot \widehat{V_\varphi \phi} \Big\rangle_{L^2(\R^{2d})}\\
            &= \int_{\R^{2d}} \hat{\mu}(x', \omega') \widehat{V_\varphi \psi}(x', \omega') \overline{\widehat{V_\varphi \phi}(x', \omega')}\,dx'\,d\omega'\\
            &= \int_{\R^{2d}} \hat{\mu}(x', \omega') \check{\psi}(\omega')\overline{\check{\phi}(\omega')} |\widehat{\varphi^*}(x')|^2 \,dx'\,d\omega' = 0
        \end{align*}
        for all $\psi, \phi \in L^2(\R^d)$, implying that $B_\mu^\varphi$ is the zero operator.
    \end{proof}

    \subsection{Non-compactness}
    As the blurring operator is based on a convolution, it should come as no surprise that it is not compact. However, a proof requires some careful considerations due to the involvement of the synthesis $V_\varphi^*$. Before proceeding with a proof, we establish two preliminary lemmas.
    \begin{lemma}\label{lemma:fourier_of_synthesis}
        Let $F \in L^1(\R^{2d})$ and $\varphi \in L^1(\R^d)$, then
        \begin{align*}
            \widehat{V_\varphi^* F}(\xi) = \int_{\R^{2d}} F(x, \omega) \hat{\varphi}(\xi - \omega)  e^{-2\pi i x\cdot(\xi - \omega)}\,dx\,d\omega.
        \end{align*}
    \end{lemma}
    \begin{proof}
        We compute
        \begin{align*}
            \widehat{V_\varphi^* F}(\xi) &= \int_{\R^d} V_\varphi^* F(t) e^{-2\pi i t \cdot \xi}\,dt\\
            &=\int_{\R^d} \left(\int_{\R^{2d}} F(x,\omega) \varphi(t-x) e^{2\pi i t \cdot \omega} \,dx\,d\omega\right) e^{-2\pi i t \cdot \xi}\,dt\\
            &= \int_{\R^{2d}} F(x, \omega) \left( \int_{\R^d} \varphi(t-x) e^{-2\pi i t\cdot(\xi - \omega)}\,dt \right)\,dx\,d\omega\\
            &= \int_{\R^{2d}} F(x, \omega) \left( \int_{\R^d} \varphi(s) e^{-2\pi i s\cdot(\xi - \omega)}\,ds \right)e^{-2\pi i x\cdot(\xi - \omega)}\,dx\,d\omega\\
            &= \int_{\R^{2d}} F(x, \omega) \hat{\varphi}(\xi - \omega)  e^{-2\pi i x\cdot(\xi - \omega)}\,dx\,d\omega
        \end{align*}
        where the exchange of order of integration is justified by Fubini.
    \end{proof}
    
    \begin{lemma}\label{lemma:modulation_iso}
        Let $F \in L^2(\R^{2d})$, then
        \begin{align*}
            \big\Vert P_{V_\varphi(L^2)} F \big\Vert_{L^2(\R^{2d})} = \big\Vert P_{V_\varphi(L^2)} F(\cdot - (0, \xi)) \big\Vert_{L^2(\R^{2d})}
        \end{align*}
        for all $\xi \in \R^{d}$.
    \end{lemma}
    \begin{proof}
        Write $F = F_1 + F_2$ where $F_1 \in V_\varphi(L^2(\R^d))$ and $F_2 \in V_\varphi(L^2(\R^d))^\perp$ for the orthogonal decomposition of $F$. The first function $F_1$ can be identified with $V_\varphi\eta$ for some $\eta \in L^2(\R^d)$. In general, we have that
        \begin{align}\label{eq:modulation_translation}
            V_\varphi(M_\xi \eta)(x, \omega) = \langle M_\xi \eta, M_\omega T_x \varphi \rangle = \langle \eta, M_{\omega - \xi}T_x \varphi \rangle = V_\varphi \eta(x, \omega - \xi).
        \end{align}
        As a consequence, $F_1(\cdot - (0, \xi)) \in V_\varphi(L^2(\R^d))$ and if we can show that $F_2(\cdot - (0, \xi)) \in V_\varphi(L^2(\R^d))^\perp$, we are done. To see this, note that for any $\chi \in L^2(\R^d)$,
        \begin{align*}
            \big\langle F_2(\cdot - (0, \xi)), V_\varphi \chi \big\rangle_{L^2(\R^{2d})} &= \big\langle F_2, V_\varphi \chi(\cdot + (0, \xi)) \big\rangle_{L^2(\R^{2d})}\\
            &= \big\langle F_2, V_\varphi (M_{-\xi}\chi) \big\rangle_{L^2(\R^{2d})} = 0
        \end{align*}
        where the final equality is due to $F_2 \in V_\varphi(L^2(\R^d))^\perp$. Combining these facts, we have that
        \begin{align*}
            \big\Vert P_{V_\varphi(L^2)} F(\cdot - (0, \xi)) \big\Vert_{L^2(\R^{2d})} &= \big\Vert P_{V_\varphi(L^2)} V_\varphi(M_\xi \eta) \big\Vert_{L^2(\R^{2d})} + \big\Vert P_{V_\varphi(L^2)} F_2(\cdot - (0,\xi)) \big\Vert_{L^2(\R^{2d})}\\
            &= \Vert \eta \Vert = \big\Vert P_{\varphi(L^2)} F \big\Vert_{L^2(\R^{2d})}
        \end{align*}
        which is what we wished to show.
    \end{proof}
    
    We are now ready to prove that the operator is non-compact through an example of a bounded sequence of functions whose image under $B_\mu^\varphi$ has no convergent subsequence.
    \begin{theorem}
        Let $\mu \in M(\R^{2d})$ and $\varphi \in M^1(\R^d)$ be such that $B_\mu^\varphi$ is not the zero operator, then $B_\mu^\varphi$ is non-compact.
    \end{theorem}
    \begin{proof}
        Since $B_\mu^\varphi$ is not the zero operator, there exists a function $\psi_0 \in L^2(\R^d)$ such that $B_\mu^\varphi \psi_0 \neq 0$. From Proposition \ref{prop:bounded_op} we know that $B_\mu^\varphi$ is a bounded operator on $L^2(\R^d)$ and so we can approximate $\psi_0$ by a function $\psi_c \in M^1(\R^d)$ with compact support in the Fourier domain such that $\Vert B_\mu^\varphi \psi_c \Vert \neq 0$. By rescaling, we can assume that $\Vert \psi_c \Vert = 1$ without loss of generality. For technical reasons, we will need to assume that both $\mu$ and $\hat{\varphi}$ also have compact support so that we can choose a positive number $R$ such that the supports of $\hat{\psi_c}$ and $\hat{\varphi}$ are contained in the $\R^d$ ball of radius $R$, $B_R$, and the support of $\mu$ is contained in the ball of the same radius in $\R^{2d}$.
        
        Consider the bounded $L^2(\R^d)$ sequence $(\psi_n)_n$ defined by $\psi_n = M_{8Rn}\psi_c$. We will show that the sequence $(B_\mu^\varphi \psi_n)_n$ is uniformly separated by means of disjoint supports in the frequency domain, contradicting compactness, and lastly show that we can remove the assumption of compact supports of $\mu$ and $\hat{\varphi}$.

        For the disjointness of the supports, we first claim that if the function $\hat{\psi}$ has support in $E \subset \R^d$ and the Fourier transform of the window, $\hat{\varphi}$, has support in $B_R$, the support of the last $d$ coordinates of $V_\varphi \psi$ is contained in $E+B_R$. Indeed,
        \begin{align*}
            V_\varphi \psi(x, \omega) = \big\langle \hat{\psi}, T_\omega M_{-x} \hat{\varphi} \big\rangle = \int_{E} \hat{\psi}(\xi) \overline{\hat{\varphi}(\xi-\omega)}e^{2\pi i (\xi-\omega) \cdot x}\,d\xi
        \end{align*}
        where the relation $V_\varphi \psi(x, \omega) = \langle \hat{\psi}, T_\omega M_{-x} \hat{\varphi} \rangle$ can be found in \cite[Lemma 3.1.1]{grochenig_book}. Note now that the statement $V_\varphi(x, \omega) = 0$ if $\omega \not\in E + B_R$ holds if $\xi \in E, \omega \not \in E+B_R \implies \xi-\omega \not\in B_R$ which in turn is easily verified.

        Next since the support of $\mu$ is contained in $B_R$, this time a $2d$-dimensional ball, we can conclude that convolution $\mu * V_\varphi \psi$ has support in the last $d$ variables contained in $E + B_{2R}$.

        Finally for the synthesized signals $B_\mu^\varphi \psi_n$, we claim that if $F \in L^1(\R^{2d})$ has support in the last $d$ variables contained in $E + B_{2R}$, it follows that the support of $\widehat{V_\varphi^* F}$ is contained in $E + B_{3R}$. To see this, we use Lemma \ref{lemma:fourier_of_synthesis} and expand as
        \begin{align*}
            \widehat{V_\varphi^* F}(\xi) &= \int_{\R^{2d}} F(x, \omega) \hat{\varphi}(\xi - \omega)  e^{-2\pi i x\cdot(\xi - \omega)}\,dx\,d\omega\\
            &= \int_{\R^d} \left(\int_{E + B_{2R}} F(x,\omega) \hat{\varphi}(\xi-\omega) e^{-2\pi i x\cdot (\xi-\omega)}\,d\omega\right)\,dx
        \end{align*}
        We claim that this quantity is zero whenever $\xi \not\in E+B_{3R}$. Indeed, $\omega \in E + B_{2R},\,\xi \not\in E + B_{3R} \implies \xi-\omega \not\in B_R$ as is easily verified.

        By these three steps we have shown that the support of the Fourier transform of $B_\mu^\varphi \psi_n = V_\varphi^* (\mu * V_\varphi \psi_n)$ is contained in $\supp{\psi_n} + B_{3R}$ and by the construction of $\psi_n$, these supports are disjoint for different $n$. Consequently, $\Vert B_\mu^\varphi \psi_n - B_\mu^\varphi \psi_m \Vert = \Vert B_\mu^\varphi \psi_n \Vert + \Vert B_\mu^\varphi \psi_m \Vert$ whenever $n \neq m$. We claim that this norm is unchanged when $\psi$ is modulated. Indeed, as we saw in \eqref{eq:modulation_translation}, a modulation of the signal corresponds to a translation in phase space and convolutions respect translations. Hence the uniform boundedness follows from Lemma \ref{lemma:modulation_iso} since
        \begin{align*}
            \Vert B_\mu^\varphi \psi \Vert = \Vert P_{V_\varphi(L^2)} (\mu * V_\varphi \psi) \Vert_{L^2(\R^{2d})}
        \end{align*}
        by Moyal's formula \eqref{eq:moyal} and \eqref{eq:gabor_proj}. This means that
        \begin{align*}
            \Vert B_\mu^\varphi \psi_n - B_\mu^\varphi \psi_m \Vert = 2 \Vert B_\mu^\varphi \psi_c \Vert > 0
        \end{align*}
        for all $n \neq m$ and there can be no convergent subsequence of $(B_\mu^\varphi \psi_n)_n$.

        Now to lift the assumption of compactness of the supports of $\mu$ and $\hat{\varphi}$, we write $\mu = \mu_0 + \mu_1$ and $\varphi = \varphi_0 + \varphi_1$ where $\mu_0$ and $\hat{\varphi_0}$ have compact support and 
        \begin{align*}
        \Vert \mu_1 \Vert_{M(\R^{2d})} &< \min\left\{ \frac{\Vert B_{\mu_0}^{\varphi_0} \psi_c \Vert}{3 \Vert \varphi_0 \Vert^2},\, \sqrt{\frac{\Vert B_{\mu_0}^{\varphi_0} \psi_c \Vert}{3}} \right\},\\
        \Vert \varphi_1 \Vert^2 &< \min\left\{ \frac{\Vert B_{\mu_0}^{\varphi_0} \psi_c \Vert}{3\Vert \mu_0 \Vert_{M(\R^{2d})}},\, \sqrt{\frac{\Vert B_{\mu_0}^{\varphi_0} \psi_c \Vert}{3}} \right\}.
        \end{align*}
        We can then decompose the operator as
        \begin{align*}
            B_\mu^\varphi = B_{\mu_0}^{\varphi_0} + \underbrace{B_{\mu_0}^{\varphi_1} + B_{\mu_1}^{\varphi_0} + B_{\mu_1}^{\varphi_1}}_{=C}.
        \end{align*}
        Now by applying the above proof to $B_{\mu_0}^{\varphi_0}$, we get a sequence $(\psi_n)_n$, all with norm $1$, such that
        \begin{align*}
            \Vert B_{\mu_0}^{\varphi_0} \psi_n - B_{\mu_0}^{\varphi_0} \psi_m \Vert = 2 \Vert B_{\mu_0}^{\varphi_0} \psi_c \Vert.
        \end{align*}
        From here, we can apply the triangle inequality to find that
        \begin{align}\nonumber
            \Vert B_\mu^\varphi (\psi_n - \psi_m) \Vert &\geq \Vert B_{\mu_0}^{\varphi_0}(\psi_n - \psi_m) \Vert - \Vert C(\psi_n - \psi_m) \Vert\\\label{eq:uniform_sep_non_comp}
            &\geq 2\big(\Vert B_{\mu_0}^{\varphi_0} \psi_c \Vert - \Vert C \Vert_{B(L^2(\R^d))}\big)
        \end{align}
        and so it will follow that $(B_\mu^\varphi \psi_n)_n$ is uniformly bounded if we can bound the norm of $C$ from above by $\Vert B_{\mu_0}^{\varphi_0} \psi_c \Vert$. Indeed, by the bound in Proposition \ref{prop:bounded_op} and the conditions imposed on the norms of $\mu_1$ and $\varphi_1$ above,
        \begin{align*}
            \Vert C \Vert_{B(L^2(\R^d))} &\leq \Vert B_{\mu_0}^{\varphi_1}\Vert_{B(L^2(\R^d))} + \Vert B_{\mu_1}^{\varphi_0}\Vert_{B(L^2(\R^d))} + \Vert B_{\mu_1}^{\varphi_1} \Vert_{B(L^2(\R^d))}\\
            &< 3 \frac{\Vert B_{\mu_0}^{\varphi_0}\Vert}{3} = \Vert B_{\mu_0}^{\varphi_0}\Vert.
        \end{align*}
        Consequently, the right hand side of \eqref{eq:uniform_sep_non_comp} is positive and $(B_\mu^\varphi \psi_n)_n$ is uniformly separated, implying non-compactness.
    \end{proof}
    
    \section{Implementations and applications}\label{sec:numerical_implementation}
    \subsection{Examples}\label{sec:examples}
    We briefly discuss some specifics of implementing the time-frequency blurring operator as well as spectrogram blurring. The code used to produce the figures in this section can be found on GitHub\footnote{\url{https://www.github.com/SimonHalvdansson/Time-Frequency-Blurring-Operator}} where details around the kernels, windows and specific settings also are available.
    
    \subsubsection{STFT blurring}\label{sec:stft_blur}
    The purest implementation of \eqref{eq:main_def} consists of simply a STFT, a convolution operation and an inverse STFT. In code, we can implement this as a function which acts on the waveform of the signal and this is what we use as the augmentation operation in Section \ref{sec:data_aug}. For visualizing the result, we can use either the regular spectrogram or a log-mel spectrogram. Both versions can be seen in Figure \ref{fig:stft-conv_overview} and another example without mel rescaling is illustrated in Figure \ref{fig:non-mel-blur-op}.

    In general, the discrete time-frequency blurring operator significantly reduces the $\ell^2$ energy of the input waveform since the phase of the STFT varies rapidly. In the figures we have performed $0-1$ normalization on the spectrograms to compensate for this phenomenon.
    
    \begin{figure}
        \centering
        \includegraphics[width=\linewidth]{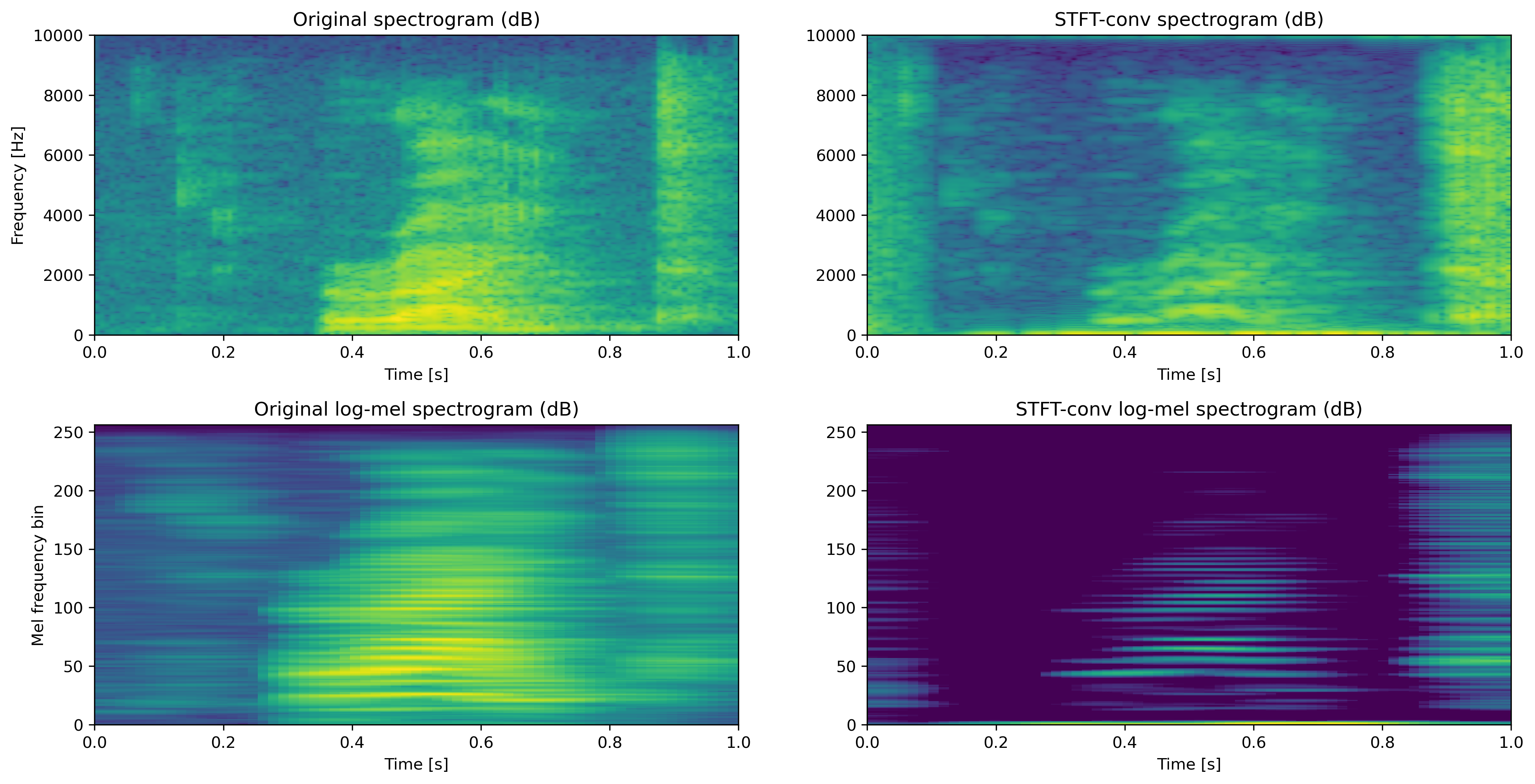}
        \caption{Spectrograms and log-mel spectrograms of an audio recording and the same recording with a time-frequency blurring operator with Gaussian kernel applied to it.}
        \label{fig:stft-conv_overview}
    \end{figure}

    \subsubsection{Spectrogram blurring}\label{sec:spec_blur}
    Spectrogram blurring as discussed in Section \ref{sec:spectrogram_blurring} is implemented by first computing the spectrogram, rescaling it to logarithmic decibel scale, and then applying a convolution. In Figure \ref{fig:spec-blur_overview}, we illustrate this on the spectrogram of an audio clip in both normal and mel scale. It is the mel-rescaled version which we use for data augmentation in Section \ref{sec:data_aug}.
    \begin{figure}
        \centering
        \includegraphics[width=\linewidth]{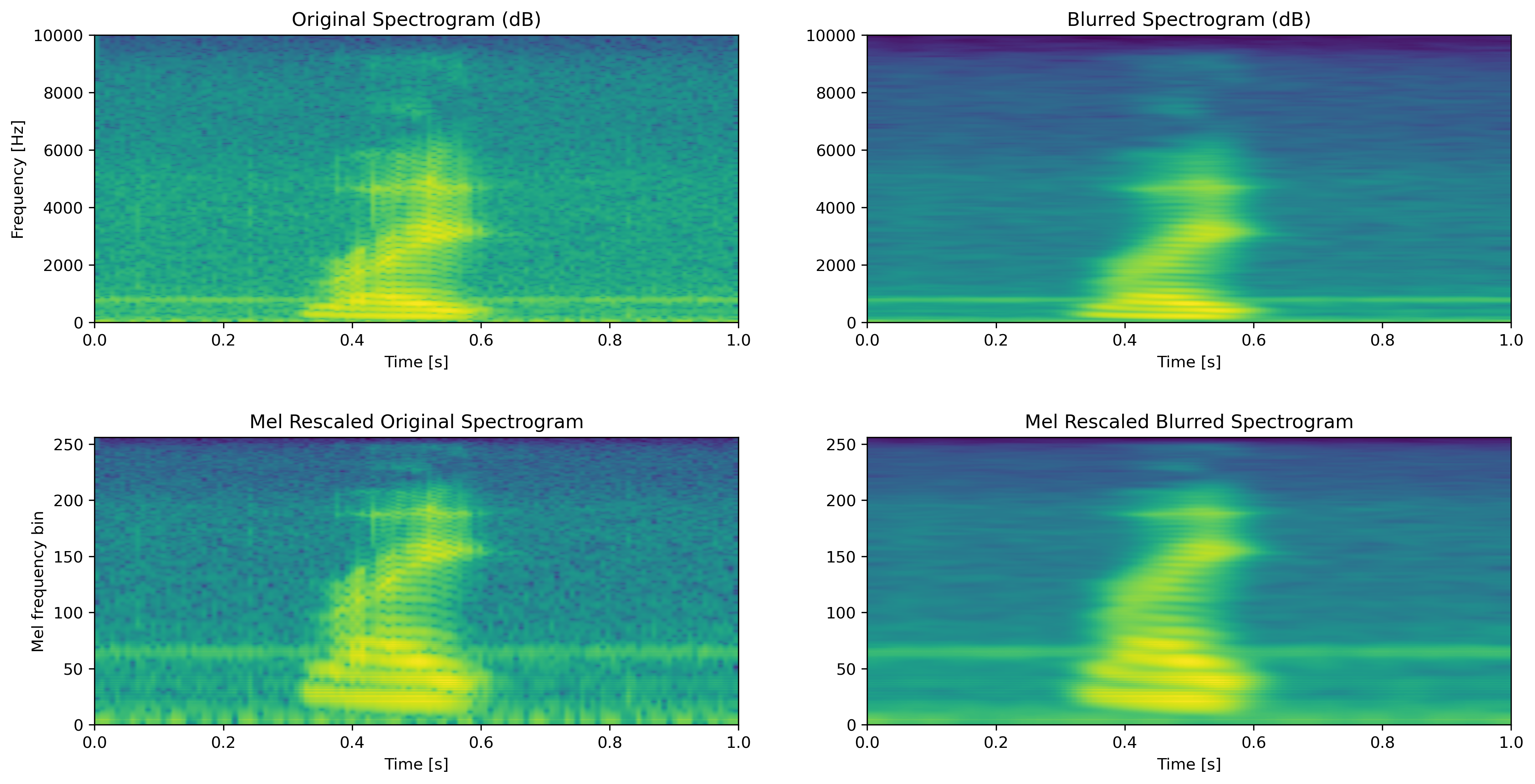}
        \caption{Spectrograms and log-mel spectrograms of an audio recording and the same spectrograms blurred with a Gaussian kernel.}
        \label{fig:spec-blur_overview}
    \end{figure}

    \subsection{Evaluation as a data augmentation method}\label{sec:data_aug}
    To investigate the performance of the blurring operator as well as spectrogram blurring for data augmentation, a convolutional neural network (CNN) using the ResNet-34 architecture \cite{He2016} and a vision transformer (ViT) using the TinyViT-11M architecture \cite{Wu2022} was trained on the SpeechCommands V2 dataset \cite{speechcommandsv2} for 35-class classification using different augmentation setups. These two models were chosen as they represent the two main contemporary paradigms for image recognition systems \cite{Dorfler2018, Gong2021}. The full code used to produce all the plots and tables of this section is available on GitHub\footnote{\url{https://www.github.com/SimonHalvdansson/Time-Frequency-Blurring-Operator}}.

    \begin{figure}\label{fig:data_augment_overview}
        \centering
        \includegraphics[width=\linewidth]{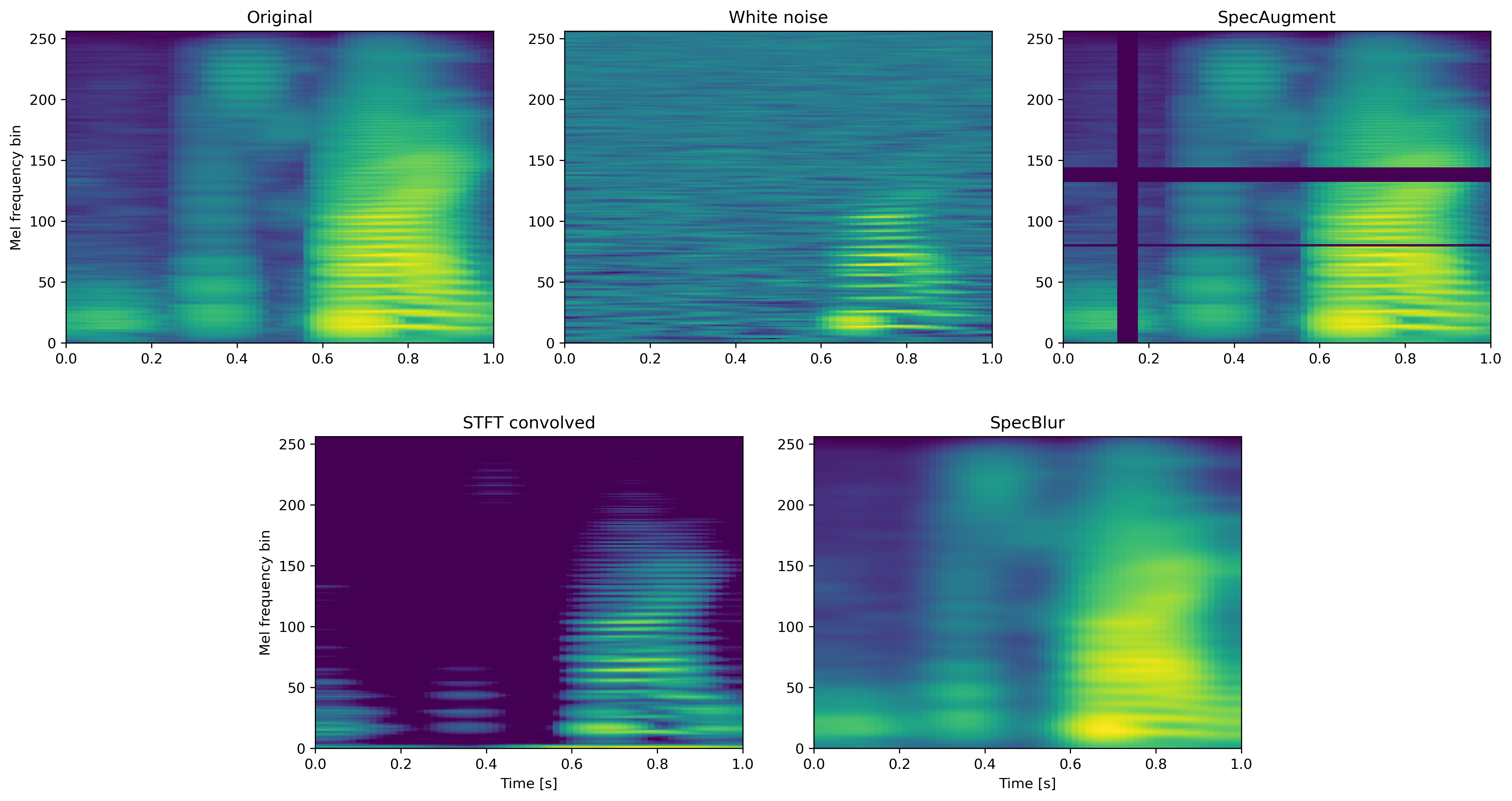}
        \caption{Log-mel spectrograms of an audio recording from the SpeechCommands V2 dataset \cite{speechcommandsv2} with different augmentation techniques applied to it.}
    \end{figure}

    \subsubsection{Setup}
    The SpeechCommands V2 dataset \cite{speechcommandsv2} contains audio recordings of utterances of 35 different words which we aim to classify using a neural network. As our focus is on data augmentation, we limit the number of training examples available to the model to learn from to 100, 300, 600 or 1000 recordings from each class. We use an $80/20$ train/validation split and use the validation set to decide when to stop training. We always use 200 examples from each class as the test data set.

    For preprocessing, all recordings are padded to be 1 second long and then the log-mel spectrogram of resolution $63 \times 256$ (time $\times$ frequency) is computed and used as input to the network. As the purpose is to evaluate the effect of using STFT convolving and spectrogram blurring for data augmentation, we also use different augmentation methods to compare against. Specifically, we consider the augmentation techniques of adding white noise to the waveform and performing basic time and frequency masking as described in the SpecAugment paper \cite{Park2019}. Finally for each of the different counts of input examples, we evaluate the model with no augmentation, all augmentations, once for each augmentation, once with only white noise and SpecAugment and once with both blurring augmentations, making for a total of 8 reported accuracies per input size. 
    
    There is a natural variation to the test accuracies as we pick the dataset which we train on randomly and randomly initialize the weights. We therefore train the network several times with different randomly selected datasets each times. In order to compare augmentation methods, we estimate the mean of the accuracies for each configuration. Here the \emph{standard error} quantifies our uncertainty in the mean and is given by $\sigma/\sqrt{n}$ where $\sigma$ is the natural standard deviation of the distribution we sample from and $n$ is the number of samples. We estimate $\sigma$ as the standard deviation of the accuracies we record over the $n$ training runs. For each configuration the training and test procedures are repeated until the estimated standard errors are so small that we can compare the accuracies of different augmentation methods. 

    For both the time-frequency blurring operator and spectrogram blurring, we use a two-dimensional Gaussian kernel with manually fine-tuned spread in time and frequency, see the code for exact details. Spectrogram blurring is performed on the log-mel spectrogram as in Section \ref{sec:spec_blur} while the time-frequency blurring operator is applied with a non-rescaled STFT as in Section \ref{sec:stft_blur}. We stress that the exact details of the implementation should be considered secondary as we are mainly interested in comparing the novel augmentation to the baseline and classical augmentation techniques. 
    
    \subsubsection{CNN results}
    We summarize the results from 1780 runs with the ResNet-34 CNN \cite{He2016} in Table \ref{table:cnn_results} which can be reproduced using the code in the GitHub repository.
    \begin{table}[H]
        \caption{Average CNN test accuracies with standard errors (\%) for 100, 300, 600 and 1000 input examples per class for different augmentation setups.}
        \begin{tabular}{l|llll}
            \toprule
            \textbf{Augmentation}   & \textbf{Acc-100}  & \textbf{Acc-300}  & \textbf{Acc-600}  & \textbf{Acc-1000}\\ \midrule
None & $47.05 \pm 0.26$ & $75.13 \pm 0.38$ & $83.13 \pm 0.25$ & $87.27 \pm 0.19$\\
White noise & $62.67 \pm 0.48$ & $78.10 \pm 0.27$ & $84.74 \pm 0.33$ & $88.77 \pm 0.21$\\
SpecAugment & $67.70 \pm 0.41$ & $80.77 \pm 0.19$ & $85.86 \pm 0.18$ & $88.99 \pm 0.16$\\
STFT-blur & $63.31 \pm 0.50$ & $78.69 \pm 0.29$ & $84.90 \pm 0.27$ & $88.88 \pm 0.18$\\
SpecBlur & $64.50 \pm 0.45$ & $80.38 \pm 0.20$ & $85.83 \pm 0.25$ & $88.91 \pm 0.29$\\
White noise + SpecAug & $66.57 \pm 0.34$ & $81.62 \pm 0.19$ & $87.38 \pm 0.19$ & $89.90 \pm 0.14$\\
STFT-blur + SpecBlur & $68.24 \pm 0.30$ & $81.67 \pm 0.19$ & $86.45 \pm 0.20$ & $89.36 \pm 0.15$\\
All & $70.48 \pm 0.23$ & $83.17 \pm 0.15$ & $87.98 \pm 0.11$ & $90.51 \pm 0.14$\\\bottomrule 
        \end{tabular}
        \label{table:cnn_results}
    \end{table}
    We see the greatest improvements compared to the baseline for the lower example configurations which is to be expected. Notably, there is a significant improvement for both STFT-blur and SpecBlur compared to the baseline for all training sizes ($p < 0.001$) as well as a significant improvement going from white noise + SpecAugment to all augmentations ($p < 0.05$). Combining STFT-blur with SpecBlur offers modest improvements compared to just using one of them. However, this could be because the augmentation is applied to each example in the training set leading to some overfitting in the direction of blurred spectrograms. 

    \subsubsection{ViT results}
    The experiments was repeated with a TinyViT-11M vision transformer \cite{Wu2022} and the results of 1310 runs are in Table \ref{table:vit_results}. The code to produce the table can be found in the main GitHub repository.
    \begin{table}[H]
        \caption{Average ViT test accuracies with standard errors (\%) for 100, 300, 600 and 1000 input examples per class for different augmentation setups.}
        \begin{tabular}{l|llll}
            \toprule
            \textbf{Augmentation}   & \textbf{Acc-100}  & \textbf{Acc-300}  & \textbf{Acc-600}  & \textbf{Acc-1000}\\ \midrule
None & $25.85 \pm 0.29$ & $71.15 \pm 0.46$ & $84.32 \pm 0.23$ & $89.17 \pm 0.20$\\
White noise & $41.64 \pm 0.32$ & $80.84 \pm 0.22$ & $87.94 \pm 0.15$ & $90.72 \pm 0.09$\\
SpecAugment & $46.97 \pm 0.33$ & $81.26 \pm 0.22$ & $87.55 \pm 0.08$ & $90.61 \pm 0.14$\\
STFT-blur & $50.46 \pm 0.28$ & $81.00 \pm 0.24$ & $87.56 \pm 0.18$ & $90.40 \pm 0.15$\\
SpecBlur & $52.67 \pm 0.30$ & $84.08 \pm 0.14$ & $89.00 \pm 0.12$ & $91.29 \pm 0.13$\\
White noise + SpecAug & $56.61 \pm 0.33$ & $84.46 \pm 0.15$ & $89.61 \pm 0.15$ & $91.80 \pm 0.15$\\
STFT-blur + SpecBlur & $67.54 \pm 0.29$ & $85.65 \pm 0.16$ & $89.22 \pm 0.17$ & $91.72 \pm 0.12$\\
All & $73.38 \pm 0.19$ & $86.89 \pm 0.14$ & $90.60 \pm 0.13$ & $92.70 \pm 0.08$\\\bottomrule 
        \end{tabular}
        \label{table:vit_results}
    \end{table}
    As vision transformers generally are more sensitive to the amount of training data, we see larger accuracy gains from using augmentation than in the CNN case. The improvement over baseline for STFT-blur and SpecBlur is statistically significant for all training sizes ($p < 0.001$). STFT-blur and SpecBlur also provide statistically significant improvements when comparing white noise together with SpecAugment against all augmentation methods combined ($p < 0.001$). While STFT-blur + SpecBlur consistently outperformed SpecBlur alone, when only using one of the two augmentation methods, SpecBlur performed better ($p < 0.001$).

    \subsubsection{Conclusions}
    We have seen that both STFT-blur and SpecBlur are promising augmentation techniques that can be used to improve the performance of both convolutional neural networks and vision transformers on spectrograms, at least at smaller scales. The computationally more efficient SpecBlur generally performs better than STFT-blur alone but combining several augmentation methods results in superior performance.

    \subsection*{Acknowledgements}
    This project was partially supported by the Project Pure Mathematics in Norway, funded by Trond Mohn Foundation and Tromsø Research Foundation.

    \printbibliography

@book{grochenig_book,
	doi = {10.1007/978-1-4612-0003-1},
	year = {2001},
	publisher = {Birkh\"{a}user Boston},
	author = {K. Gr\"{o}chenig},
	title = {Foundations of Time-Frequency Analysis}
}

@article{Balazs2024,
  doi = {10.1016/j.jmaa.2023.127579},
  year = {2024},
  publisher = {Elsevier {BV}},
  volume = {529},
  number = {1},
  pages = {127579},
  author = {Peter Balazs and Federico Bastianoni and Elena Cordero and Hans G. Feichtinger and Nina Schweighofer},
  title = {Comparisons between {Fourier} and {STFT} multipliers: The smoothing effect of the short-time {Fourier} transform},
  journal = {J. Math. Anal. Appl.}
}

@article{Luef2018,
  doi = {10.1016/j.matpur.2017.12.004},
  year = {2018},
  publisher = {Elsevier {BV}},
  volume = {118},
  pages = {288--316},
  author = {Franz Luef and Eirik Skrettingland},
  title = {Convolutions for localization operators},
  journal = {J. Math. Pures Appl.}
}

@article{Luef2019,
  doi = {10.1007/s00041-019-09663-3},
  year = {2019},
  publisher = {Springer Science and Business Media {LLC}},
  volume = {25},
  number = {4},
  pages = {2064--2108},
  author = {Franz Luef and Eirik Skrettingland},
  title = {Mixed-State Localization Operators: Cohen's Class and Trace Class Operators},
  journal = {J. Fourier Anal. Appl.}
}

@article{Luef2019_acc,
  doi = {10.1007/s00365-019-09465-2},
  year = {2019},
    journal = {Constr. Approx.},
  publisher = {Springer Science and Business Media {LLC}},
  volume = {52},
  number = {1},
  pages = {31--64},
  author = {Franz Luef and Eirik Skrettingland},
  title = {On Accumulated {C}ohen's Class Distributions and Mixed-State Localization Operators}
}

@article{Werner1984,
  doi = {10.1063/1.526310},
  year = {1984},
  publisher = {{AIP} Publishing},
  volume = {25},
  number = {5},
  pages = {1404--1411},
  author = {R. Werner},
  title = {Quantum harmonic analysis on phase space},
  journal = {Journal of Mathematical Physics}
}

@article{FEICHTINGER1989307,
title = {Banach spaces related to integrable group representations and their atomic decompositions, {I}},
journal = {J. Funct. Anal.},
volume = {86},
number = {2},
pages = {307-340},
year = {1989},
doi = {10.1016/0022-1236(89)90055-4},
author = {Hans G Feichtinger and K.H Gröchenig},
}

@article{Luef2021,
  doi = {10.1016/j.jfa.2020.108883},
  year = {2021},
  publisher = {Elsevier {BV}},
  volume = {280},
  number = {6},
  pages = {108883},
  author = {Franz Luef and Eirik Skrettingland},
  title = {A {Wiener} {Tauberian} theorem for operators and functions},
  journal = {J. Funct. Anal.}
}

@article{daubechies1988_loc,
  doi = {10.1109/18.9761},
  year = {1988},
  publisher = {Institute of Electrical and Electronics Engineers ({IEEE})},
  volume = {34},
  number = {4},
  pages = {605--612},
  author = {I. Daubechies},
  title = {Time-frequency localization operators: a geometric phase space approach},
  journal = {IEEE Trans. Inform. Theory}
}

@article{Cordero2003,
  doi = {10.1016/s0022-1236(03)00166-6},
  year = {2003},
  publisher = {Elsevier {BV}},
  volume = {205},
  number = {1},
  pages = {107--131},
  author = {Elena Cordero and Karlheinz Gr\"{o}chenig},
  title = {Time{\textendash}Frequency analysis of localization operators},
  journal = {J. Funct. Anal.}
}

@incollection{CORDERO2007,
  doi = {10.1142/9789812770707_0005},
  year = {2007},
  publisher = {World {S}cientific},
  pages = {83--110},
  author = {Elena Cordero and Luigi Rodino and Karlheinz Gr\"{o}chenig},
  title = {LOCALIZATION OPERATORS AND TIME-FREQUENCY ANALYSIS},
  booktitle = {Harmonic,  Wavelet and P-Adic Analysis}
}

@book{Wong2002,
  doi = {10.1007/978-3-0348-8217-0},
  year = {2002},
  publisher = {Birkh\"{a}user Basel},
  author = {M. W. Wong},
  title = {Wavelet Transforms and Localization Operators}
}

@article{Guo2015,
  doi = {10.1007/s00041-015-9424-z},
  year = {2015},
  publisher = {Springer Science and Business Media {LLC}},
  volume = {22},
  number = {2},
  pages = {427--461},
  author = {Weichao Guo and Dashan Fan and Huoxiong Wu and Guoping Zhao},
  title = {Sharpness of Complex Interpolation on $\alpha$-Modulation Spaces},
  journal = {J. Fourier Anal. Appl.}
}

@InProceedings{papakipos2022augly,
    author    = {Papakipos, Zo\"e and Bitton, Joanna},
    title     = {AugLy: Data Augmentations for Adversarial Robustness},
    booktitle = {Proceedings of the IEEE/CVF Conference on Computer Vision and Pattern Recognition (CVPR) Workshops},
    year      = {2022},
    pages     = {156-163},
    doi       = {10.1109/CVPRW56347.2022.00027}
}

@inproceedings{Park2019,
  doi = {10.21437/interspeech.2019-2680},
  year = {2019},
  publisher = {{ISCA}},
  author = {Daniel S. Park and William Chan and Yu Zhang and Chung-Cheng Chiu and Barret Zoph and Ekin D. Cubuk and Quoc V. Le},
  title = {{SpecAugment}: A Simple Data Augmentation Method for Automatic Speech Recognition},
  booktitle = {Interspeech 2019}
}

@inproceedings{McFee2015,
    title = "A software framework for musical data augmentation",
    author  = {McFee, B. and Humphrey, E.J. and Bello, J.P.},
    year = "2015",
    pages = "248--254",
    booktitle = {16th International Society for Music Information Retrieval Conference},
    series  = {ISMIR}
}

@misc{audiomentations,
    title = "Audiomentations",
    author = "Iver Jordal",
    year = {2024},
    publisher = {GitHub},
    journal = {GitHub repository},
    howpublished = {\url{https://github.com/iver56/audiomentations}}
}

@misc{sigment,
    title = "Sigment",
    author = "Edwin Onuonga",
    year = {2020},
    publisher = {GitHub},
    journal = {GitHub repository},
    howpublished = {\url{https://github.com/eonu/sigment}}
}

@misc{nlpaug,
    title = "{nlpaug}",
    author = "Edward Ma",
    year = {2022},
    publisher = {GitHub},
    journal = {GitHub repository},
    howpublished = {\url{https://github.com/makcedward/nlpaug}}
}

@INPROCEEDINGS{wavaugment2020,
  author={Kharitonov, Eugene and Rivière, Morgane and Synnaeve, Gabriel and Wolf, Lior and Mazaré, Pierre-Emmanuel and Douze, Matthijs and Dupoux, Emmanuel},
  booktitle={2021 IEEE Spoken Language Technology Workshop (SLT)}, 
  title={Data Augmenting Contrastive Learning of Speech Representations in the Time Domain}, 
  year={2021},
  pages={215-222},
  doi={10.1109/SLT48900.2021.9383605}}

@article{Rommel2022,
  doi = {10.1088/1741-2552/aca220},
  year = {2022},
  publisher = {{IOP} Publishing},
  volume = {19},
  number = {6},
  pages = {066020},
  author = {C{\'{e}}dric Rommel and Joseph Paillard and Thomas Moreau and Alexandre Gramfort},
  title = {Data augmentation for learning predictive models on {EEG}: a systematic comparison},
  journal = {J. Neural Eng.}
}

@INPROCEEDINGS{Wang2019,
  author={Wang, Jisung and Kim, Sangki and Lee, Yeha},
  booktitle={ICASSP 2019 - 2019 IEEE International Conference on Acoustics, Speech and Signal Processing (ICASSP)}, 
  title={Speech {Augmentation} {Using} {Wavenet} in {Speech} {Recognition}}, 
  year={2019},
  pages={6770-6774},
  doi={10.1109/ICASSP.2019.8683388}}

@inproceedings{Wang2021spec,
  doi = {10.21437/interspeech.2021-140},
  year = {2021},
  publisher = {{ISCA}},
  author = {Helin Wang and Yuexian Zou and Wenwu Wang},
  title = {{SpecAugment}++: A Hidden Space Data Augmentation Method for Acoustic Scene Classification},
  booktitle = {Interspeech 2021}
}

@article{Wei2020,
  doi = {10.1088/1742-6596/1453/1/012085},
  year = {2020},
  publisher = {{IOP} Publishing},
  volume = {1453},
  number = {1},
  pages = {012085},
  author = {Shengyun Wei and Shun Zou and Feifan Liao and Weimin Lang},
  title = {A Comparison on Data Augmentation Methods Based on Deep Learning for Audio Classification},
  journal = {Journal of Physics: Conference Series}
}

@article{raju2018data,
    title={Data augmentation for robust keyword spotting under playback interference},
    author={Raju, Anirudh and Panchapagesan, Sankaran and Liu, Xing and Mandal, Arindam and Strom, Nikko},
    year={2018},
    eprint={1808.00563},
    archivePrefix={arXiv},
    primaryClass={cs.CL}
}

@article{Griffin1984,
  title = {Signal estimation from modified short-time {Fourier} transform},
  volume = {32},
  DOI = {10.1109/tassp.1984.1164317},
  number = {2},
  journal = {IEEE Trans. Acoust., Speech, Signal Process.},
  publisher = {Institute of Electrical and Electronics Engineers (IEEE)},
  author = {Griffin,  D. and Jae Lim},
  year = {1984},
  pages = {236–243}
}

@article{AbayomiAlli2022,
  doi = {10.3390/electronics11223795},
  year = {2022},
  publisher = {{MDPI} {AG}},
  volume = {11},
  number = {22},
  pages = {3795},
  author = {Olusola O. Abayomi-Alli and Robertas Dama{\v{s}}evi{\v{c}}ius and Atika Qazi and Mariam Adedoyin-Olowe and Sanjay Misra},
  title = {Data Augmentation and Deep Learning Methods in Sound Classification: A Systematic Review},
  journal = {Electronics}
}

@article{Zhang2020,
    Author = {Yu Zhang and James Qin and Daniel S. Park and Wei Han and Chung-Cheng Chiu and Ruoming Pang and Quoc V. Le and Yonghui Wu},
    Title = "{Pushing the Limits of Semi-Supervised Learning for Automatic Speech Recognition}",
    Year = {2020},
    eprint={2010.10504},
    archivePrefix={arXiv},
    primaryClass={eess.AS}
}

@InProceedings{Chen2022,
  title = 	 {{BEAT}s: Audio Pre-Training with Acoustic Tokenizers},
  author =       {Chen, Sanyuan and Wu, Yu and Wang, Chengyi and Liu, Shujie and Tompkins, Daniel and Chen, Zhuo and Che, Wanxiang and Yu, Xiangzhan and Wei, Furu},
  booktitle = 	 {Proceedings of the 40th International Conference on Machine Learning},
  pages = 	 {5178--5193},
  year = 	 {2023},
  editor = 	 {Krause, Andreas and Brunskill, Emma and Cho, Kyunghyun and Engelhardt, Barbara and Sabato, Sivan and Scarlett, Jonathan},
  volume = 	 {202},
  series = 	 {Proceedings of Machine Learning Research},
  publisher =    {PMLR},
}

@incollection{Vygon2021,
  doi = {10.1007/978-3-030-87802-3_69},
  year = {2021},
  publisher = {Springer International Publishing},
  pages = {773--785},
  author = {Roman Vygon and Nikolay Mikhaylovskiy},
  title = {Learning Efficient Representations for Keyword Spotting with Triplet Loss},
  booktitle = {Speech and Computer}
}

@inproceedings{dosovitskiy2021,
	title	= {An Image is Worth 16x16 Words: Transformers for Image Recognition at Scale},
	author	= {Alexander Kolesnikov and Alexey Dosovitskiy and Dirk Weissenborn and Georg Heigold and Jakob Uszkoreit and Lucas Beyer and Matthias Minderer and Mostafa Dehghani and Neil Houlsby and Sylvain Gelly and Thomas Unterthiner and Xiaohua Zhai},
	booktitle={International Conference on Learning Representations},
	year	= {2021}
}

@article{Halvdansson2023,
  doi = {10.1016/j.jfa.2023.110096},
  year = {2023},
  publisher = {Elsevier {BV}},
  volume = {285},
  number = {8},
  pages = {110096},
  author = {Simon Halvdansson},
  title = {Quantum harmonic analysis on locally compact groups},
  journal = {J. Funct. Anal.}
}

@book{Wong1998,
  doi = {10.1007/b98973},
  year = {1998},
  publisher = {Springer-Verlag},
  series = {Universitext},
  author =  {M. W. Wong},
  title = {Weyl Transforms}
}

@inproceedings{Gong2021,
  doi = {10.21437/interspeech.2021-698},
  year = {2021},
  publisher = {{ISCA}},
  author = {Yuan Gong and Yu-An Chung and James Glass},
  title = {{AST}: {Audio} {Spectrogram} {Transformer}},
  booktitle = {Interspeech 2021}
}

@article{speechcommandsv2,
    author = { {Warden}, P.},
    title = "{Speech Commands: A Dataset for Limited-Vocabulary Speech Recognition}",
    archivePrefix = "arXiv",
    eprint = {1804.03209},
    primaryClass = "cs.CL",
    year = 2018,
}

@inproceedings{He2016,
  doi = {10.1109/cvpr.2016.90},
  year = {2016},
  publisher = {{IEEE}},
  author = {Kaiming He and Xiangyu Zhang and Shaoqing Ren and Jian Sun},
  title = {Deep Residual Learning for Image Recognition},
  booktitle = {2016 {IEEE} Conference on Computer Vision and Pattern Recognition ({CVPR})}
}

@misc{saber,
    title = "{SABER} - {Semi}-{Supervised} {Audio} {Baseline} for {Easy} {Reproduction}",
    author = "Arjun Variar",
    year = {2021},
    publisher = {GitHub},
    journal = {GitHub repository},
    howpublished = {\url{https://github.com/SABER-labs/SABER}}
}

@article{Hlawatsch1992,
  doi = {10.1109/79.127284},
  year = {1992},
  publisher = {Institute of Electrical and Electronics Engineers ({IEEE})},
  volume = {9},
  number = {2},
  pages = {21--67},
  author = {F. Hlawatsch and G.F. Boudreaux-Bartels},
  title = {Linear and quadratic time-frequency signal representations},
  journal = {IEEE Signal Process. Mag.}
}

@article{Mack2020,
  doi = {10.1109/lsp.2019.2955818},
  year = {2020},
  publisher = {Institute of Electrical and Electronics Engineers ({IEEE})},
  volume = {27},
  pages = {61--65},
  author = {Wolfgang Mack and Emanuel A. P. Habets},
  title = {Deep Filtering: Signal Extraction and Reconstruction Using Complex Time-Frequency Filters},
  journal = {{IEEE} Signal Processing Letters}
}

@inproceedings{Dong2018,
  doi = {10.1109/icassp.2018.8462506},
  year = {2018},
  publisher = {{IEEE}},
  author = {Linhao Dong and Shuang Xu and Bo Xu},
  title = {Speech-Transformer: A No-Recurrence Sequence-to-Sequence Model for Speech Recognition},
  booktitle = {2018 {IEEE} International Conference on Acoustics,  Speech and Signal Processing ({ICASSP})}
}

@book{Suppappola2018,
  year = {2018},
  publisher = {{CRC} Press},
  editor = {Antonia Papandreou-Suppappola},
  title = {Applications in Time-Frequency Signal Processing}
}

@article{Williamson_2016,
    doi = {10.1109/taslp.2015.2512042},
    year = 2016,
    publisher = {Institute of Electrical and Electronics Engineers ({IEEE})},
    volume = {24},
    number = {3},
    pages = {483--492},
    author = {Donald S. Williamson and Yuxuan Wang and DeLiang Wang},
    title = {Complex Ratio Masking for Monaural Speech Separation},
    journal = {IEEE/ACM Trans. Audio, Speech, Language Process.}
}

@article{Dorfler2009,
  title = {Representation of Operators in the Time-Frequency Domain and Generalized {Gabor} Multipliers},
  volume = {16},
  DOI = {10.1007/s00041-009-9085-x},
  number = {2},
  journal = {J. Fourier Anal. Appl.},
  publisher = {Springer Science and Business Media LLC},
  author = {D\"{o}rfler,  Monika and Torrésani,  Bruno},
  year = {2009},
  pages = {261–293}
}

@article{Radford2022,
    title={Robust Speech Recognition via Large-Scale Weak Supervision},
    author={Alec Radford and Jong Wook Kim and Tao Xu and Greg Brockman and Christine McLeavey and Ilya Sutskever},
    year={2022},
    eprint={2212.04356},
    archivePrefix={arXiv},
    primaryClass={eess.AS}
}

@inproceedings{Dorfler2007,
    title={Spreading function representation of operators and Gabor multiplier approximation},
    author={D{\"o}rfler, Monika and Torr{\'e}sani, Bruno},
    booktitle={SAMPTA’07, Sampling Theory and Applications},
    year={2007},
    address={Thessaloniki, Greece},
    note={\url{https://hal.archives-ouvertes.fr/hal-00146274}}
}

@article{Dorfler2018,
  title = {Basic filters for convolutional neural networks applied to music: Training or design?},
  volume = {32},
  DOI = {10.1007/s00521-018-3704-x},
  number = {4},
  journal = {Neural Comput. Appl.},
  publisher = {Springer Science and Business Media LLC},
  author = {D\"{o}rfler,  Monika and Grill,  Thomas and Bammer,  Roswitha and Flexer,  Arthur},
  year = {2018},
  pages = {941–954}
}

@inbook{Wu2022,
  title = {{TinyViT}: Fast Pretraining Distillation for Small Vision Transformers},
  DOI = {10.1007/978-3-031-19803-8_5},
  booktitle = {Computer Vision – ECCV 2022},
  publisher = {Springer Nature Switzerland},
  author = {Wu,  Kan and Zhang,  Jinnian and Peng,  Houwen and Liu,  Mengchen and Xiao,  Bin and Fu,  Jianlong and Yuan,  Lu},
  year = {2022},
  pages = {68–85}
}

@article{Olivero2012,
author = {Olivero, Anaïk and Torrésani, Bruno and Depalle, Philippe and Kronland-Martinet, Richard},
year = {2012},
title = {Sound morphing strategies based on alterations of time-frequency representations by {Gabor} multipliers},
journal = {Proceedings of the AES International Conference},
url = {https://hal.science/hal-00682959}
}

@inproceedings{Giorgi2022,
  title={Mel Spectrogram Inversion with Stable Pitch},
  author={Bruno Di Giorgi and Mark Levy and Richard Sharp},
  booktitle={International Society for Music Information Retrieval Conference},
  year={2022},
  url = {https://arxiv.org/pdf/2208.12782}
}

@inproceedings{Prusa2017,
  title = {Phase vocoder done right},
  DOI = {10.23919/eusipco.2017.8081353},
  booktitle = {2017 25th European Signal Processing Conference (EUSIPCO)},
  publisher = {IEEE},
  author = {Prusa,  Zdenek and Holighaus,  Nicki},
  year = {2017},
}

@inproceedings{Schroter2022,
  title = {Deepfilternet: A Low Complexity Speech Enhancement Framework for Full-Band Audio Based On Deep Filtering},
  DOI = {10.1109/icassp43922.2022.9747055},
  booktitle = {ICASSP 2022 - 2022 IEEE International Conference on Acoustics,  Speech and Signal Processing (ICASSP)},
  publisher = {IEEE},
  author = {Schroter,  Hendrik and Escalante-B,  Alberto N. and Rosenkranz,  Tobias and Maier,  Andreas},
  year = {2022},
}
    
\end{document}